\documentclass[12pt]{amsart}
\usepackage{amsmath, amsfonts, amssymb, amsthm,hyperref,mathtools,array}
\usepackage[T1]{fontenc}
\hypersetup{hypertex=true,
	colorlinks=true,
	linkcolor=blue,
	anchorcolor=blue,
	citecolor=blue}
\usepackage{bm}
\allowdisplaybreaks[4]
\def\ord{{\rm ord}}

\def\supp {\rm supp}

\def\u{{\bm u}}

\def\e{{\bm e}}
\def\0{{\bm 0}}
\def\i{{\bm i}}
\def\1{{\bf 1}}

\def\div{{\rm div}}
\def\diag{{\rm diag}}

\def\supp{{\rm supp}}

\def\pmod #1{\ ({\rm{mod}}\ #1)}

\theoremstyle{plain}
\newtheorem{theorem}{Theorem}[section]
\newtheorem{lemma}{Lemma}

\theoremstyle{definition}

\theoremstyle{remark}
\newtheorem{remark}{Remark}

\makeatletter
\@namedef{subjclassname@2020}{%
	\textup{2020} Mathematics Subject Classification}
\makeatother
\vspace{4mm}

\begin{document}
	
	\title[Cyclic permutations and Polynomial Values
	in multiplicative subgroups]
	{Cyclic permutations of Large Subsets with Polynomial Values
		in multiplicative subgroups of finite fields}
	\author[H.-L. Wu and H.-X. Ni]{Hai-Liang Wu and He-Xia Ni*}
	
	\address {(Hai-Liang Wu) School of Science, Nanjing University of Posts and Telecommunications, Nanjing 210023, People's Republic of China}
	\email{\tt whl.math@smail.nju.edu.cn}
	
	\address {(He-Xia Ni) Department of Applied Mathematics, Nanjing Audit University, Nanjing 211815, People's Republic of China}
	\email{\tt nihexia@yeah.net}

	\keywords{cyclic permutations, Hamilton cycles, character sums, finite fields.
		\newline \indent 2020 {\it Mathematics Subject Classification}. Primary 11T24, 11L40; Secondary 11T30.
		\newline \indent This research was supported by the National Natural Science Foundation of China (Grant Nos. 12671009 and 12371004).
		\newline \indent *Corresponding author.}
	
	\begin{abstract}
		Let $f(t)\in\mathbb{Z}[t]$ be a nonconstant polynomial with nonzero discriminant and let $k\ge2$ be an integer. Inspired by the work of Alon and Bourgain, for
		sufficiently large prime $p\equiv1\pmod{k}$, we study cyclic orderings of subsets
		$A\subseteq \mathbb{F}_p$ for which
		$	f(a_i+a_{i+1})$	is a nonzero $k$-th power for every consecutive pair. By combining mixed character-sum
		estimates, Fourier analysis on $\mathbb{F}_p$, and spectral graph methods, we establish a threshold $c(p,k,f)$ such that
		every subset $A$ with $\#A\ge c(p,k,f)$ admits such a cyclic ordering. We also give lower and upper bounds for the optimal threshold.

	\end{abstract}
	\maketitle
	
	\tableofcontents

	\section{Introduction}
	\setcounter{lemma}{0}
	\setcounter{theorem}{0}
	\setcounter{equation}{0}
	\setcounter{conjecture}{0}
	\setcounter{remark}{0}
	\setcounter{corollary}{0}

   \subsection{Notation} 
   
   Throughout this paper, $p$ denotes a prime. Let $\mathbb{F}_p$ be the finite field with $p$ elements and let $\overline{\mathbb{F}_p}$ be an algebraic closure of $\mathbb{F}_p$. Let $\mathbb{F}_p^*=\mathbb{F}_p\setminus\{0\}$ be the multiplicative group of nonzero elements over $\mathbb{F}_p$. The cyclic group of multiplicative characters of $\mathbb{F}_p$ is denoted by $\widehat{\mathbb{F}_p^*}$, and let $\chi_p$ be a generator of $\widehat{\mathbb{F}_p^*}$. For any character $\chi: \mathbb{F}_p^*\rightarrow\mathbb{C}^*$, we additionally define $\chi(0)=0$. Also, $\varepsilon$ denotes the trivial multiplicative character. For any nonempty subset $A\subseteq\mathbb{F}_p$, the characteristic function of $A$ is written as 
   $$1_A(x)=\begin{cases}
   	1 & \mbox{if}\ x\in A,\\
   	0 & \mbox{otherwise}.
   \end{cases}$$
  
  Given any $f(t)\in\mathbb{Z}[t]$, we use $\deg(f)$ to denote the degree of $f$. If $\deg(f)=n\ge1$, then the discriminant of $f$ is defined by 
  $$\Delta(f)=a_n^{2n-2}\prod_{1\le i<j\le n}\left(\alpha_j-\alpha_i\right)^2,$$
  where $a_n$ is the leading coefficient of $f$ and $\alpha_1,\alpha_2,\cdots,\alpha_n$ are all roots of $f$ in $\mathbb{C}$. 
  
  For any complex matrix $M$, the symbol $M(i,j)$ denotes the $(i,j)$-entry of $M$. Finally, $\#S$ denotes the cardinality of a set $S$. 
  
  Fix an integer $k\ge2$. Let $f$ and $g$ be functions defined on the set 
  $$\mathcal{P}_k=\left\{p: \text{$p$ is a prime and $p\equiv 1\pmod {k}$}\right\}.$$ 
  As usual, $f(p)\ll g(p)$ (or $f=O(g)$) means that there exist a positive real number $C$ and a real number $x_0$ such that 
  $$|f(p)|\le C\cdot g(p)$$
  for any prime $p\in\mathcal{P}_k$ with $p\ge x_0$. In addition, the symbol $f(p)\asymp g(p)$ means that $f(p)\ll g(p)$ and $g(p)\ll f(p)$.
	
	\subsection{Background and motivation}  A broad class of cyclic permutation problems asks whether the elements of a finite set can be ordered so that each pair of consecutive elements
	satisfies a prescribed compatibility condition. For example, the famous prime circle problem asks whether there exists a permutation 
	$$a_1,a_2,\cdots,a_n$$
	of $1,2,\cdots,n$ such that $a_i+a_{i+1}$ is prime for every $1\le i\le n$, where $n$ is a positive even integer and $a_{n+1}=a_1$. Such questions arise naturally in combinatorics, number theory, and graph theory, since an admissible cyclic permutation is equivalent to a Hamilton cycle in the graph whose edges represent compatible pairs. Let $G$ be a graph with $m$ vertices, where $m\ge3$. A Hamilton cycle in $G$ has the form 
	$$v_1,v_2,\cdots,v_m, v_1,$$
	where $v_1,v_2,\cdots,v_m$ are precisely all the vertices of $G$ and each consecutive pair (including $(v_{m},v_1)$) is adjacent. Many sufficient conditions are known for the existence of a Hamilton cycle in a graph. For example, Dirac \cite{Dirac} showed that if $\deg_G(v)\ge m/2$ for every vertex $v$, then $G$ has a Hamilton cycle, where $\deg_G(v)$ is the degree of the vertex $v$ in $G$. Ore \cite{Ore} showed that if $\deg_G(v_1)+\deg_G(v_2)\ge m$ for any two distinct nonadjacent vertices $v_1,v_2$, then $G$ has a Hamilton cycle. For further results on Hamilton cycles, see \cite{CE,GCS,KR}.

	On the other hand, additive combinatorics over finite fields is a very active area of research. Problems of this type are also closely related to polynomials, character sums, and algebraic curves over finite fields. These connections allow us to apply powerful tools from graph theory, algebraic geometry and number theory. For example, Alon and Bourgain \cite{AB} investigated the additive structures of multiplicative subgroups of finite fields. They \cite[Theorem 1.2]{AB} showed that there is an absolute constant $c>0$ so that for any subgroup $H\le \mathbb{F}_q^*$ with 
		$$\#H\ge c\frac{q^{3/4}(\log q)^{1/2}(\log\log\log q)^{1/2}}{\log\log q},$$
		there exists a cyclic permutation $a_1,a_2,\cdots,a_{\#H}$  of $H$ such that $a_i+a_{i+1}\in H$ for any $1\le i\le \#H$ and $a_{\#H+1}=a_1$. Also, recently Chang \cite{Chang} studied the existence of the nontrivial arithmetic progressions in multiplicative subgroups of finite fields, and Vinh \cite{V14} obtained several results on the solvability of systems of certain equations over finite fields by applying spectral graph theory. 
	
	Motivated by the above results, in this paper, we further study an arithmetic cyclic permutation problem over finite fields. Let $f\in\mathbb Z[t]$ be a fixed nonconstant polynomial, let $k\ge2$ be a fixed integer, and let $p\equiv1\pmod{k}$ be a prime. Let 
	$$D_k(\mathbb{F}_p)=\{x^k:x\in\mathbb F_p^*\}$$
	be the subgroup of nonzero $k$-th powers in $\mathbb F_p$. Given a nonempty subset $A\subseteq\mathbb F_p$, we ask whether its elements can be
	arranged cyclically as
	$$a_1,a_2,\ldots,a_{\#A}$$
	in such a way that
	$$f(a_i+a_{i+1})\in D_k(\mathbb{F}_p)$$
	for every $1\le i\le \#A$, where $a_{\#A+1}=a_1$. 
	
	This problem possesses both an algebraic and a combinatorial component. The relevant algebraic quantity is 
		\begin{equation}\label{Eq. definition of N(p,k,f)}
		N(p,k,f)=\#\left\{x\in\mathbb{F}_p: f(x)\in D_k(\mathbb{F}_p)\right\}.
	\end{equation}
	In Lemma \ref{Lem. asymptotic formula for N(p,k,f)}, with the help of Weil's bound, we will see that $N(p,k,f)$ has an asymptotic formula 
	$$ N(p,k,f)=\frac{p}{k}+O_{k,f}(\sqrt{p}).$$
	For the combinatorial component of this problem, we naturally introduce the graph $G_p=G_{p,k,f}$ with
	vertex set $\mathbb F_p$, in which two distinct vertices $x,y$ are adjacent if and only if
	$$f(x+y)\in D_k(\mathbb{F}_p).$$
	An admissible cyclic permutation of $A$ is then precisely a Hamilton
	cycle in the subgraph $H_A$ of $G_p$ induced by $A$.

    Our main quantitative objective is to determine a universal threshold
    $c(p,k,f)$ such that every subset $A\subseteq\mathbb{F}_p$ with $\#A\ge c(p,k,f)$ admits the required cyclic permutation.
    
	\subsection{Main results} 
	
	Set 
	\begin{equation}\label{Eq. definition of Lambda(p,k,f)}
		\Lambda(p,k,f)=1+\frac{\deg(f)}{k}(1+(k-1)\sqrt{p}),
	\end{equation}
	and the threshold 
	$$c(p,k,f)=p+2-N(p,k,f)+\left\lfloor\frac{2p\Lambda(p,k,f)}{N(p,k,f)}\right\rfloor.$$
	Now we state our first theorem. 
	
	\begin{theorem}\label{Thm. A}
		Let $f(t)\in\mathbb{Z}[t]$ be a nonconstant polynomial with $\Delta(f)\neq 0$ and let $k\ge2$ be an integer. Let $p\equiv 1\pmod {k}$ be a sufficiently large prime. Then, for any subset $A\subseteq\mathbb{F}_p$ with $\#A\ge c(p,k,f)$, there exists a permutation $a_1, a_2,\cdots, a_{\#A}$ of $A$ such that 
			$$f(a_1+a_2), f(a_2+a_3), \cdots, f(a_{\#A-1}+a_{\#A}), f(a_{\#A}+a_1)\in D_k(\mathbb{F}_p).$$
	\end{theorem}
	
	\begin{remark}\label{Rem. of Thm. A}
		The significance of this result is that the conclusion holds uniformly for every subset with cardinality greater than the threshold $c(p,k,f)$, rather than merely for the full field or for a typical subset. 
	\end{remark}
	
Under the notation introduced above, for every sufficiently large prime $p\equiv 1\pmod{k}$, let $c_{\min}(p,k,f)$ denote the smallest integer such that any subset $A\subseteq\mathbb{F}_p$ with $\#A\ge c_{\min}(p,k,f)$ has a permutation $a_1, a_2,\cdots, a_{\#A}$ of $A$ satisfying 
	$$f(a_i+a_{i+1})\in D_k(\mathbb{F}_p)$$
	for any $1\le i \le \#A$, where $a_{\#A+1}=a_1$. Now we state our second theorem.
	
	\begin{theorem}\label{Thm. B}
		Let $f(t)\in\mathbb{Z}[t]$ be a nonconstant polynomial with $\Delta(f)\neq 0$ and let $k\ge2$ be an integer. Let $p\equiv 1\pmod {k}$ be a sufficiently large prime. Then
		$$p+2-N(p,k,f)\le c_{\min}(p,k,f)\le p+2-N(p,k,f)+\left\lfloor\frac{2p\Lambda(p,k,f)}{N(p,k,f)}\right\rfloor.$$
		Moreover, we have the asymptotic formula
		$$c_{\min}(p,k,f)=\left(1-\frac{1}{k}\right)p+O_{k,f}(\sqrt{p}).$$
	\end{theorem}
	
	\subsection{Outline of the paper} Section 2 is devoted to the proof of Theorem \ref{Thm. A}. We shall prove Theorem \ref{Thm. B} in Section 3.

	\section{Proof of Theorem \ref{Thm. A}}
	\setcounter{lemma}{0}
	\setcounter{theorem}{0}
	\setcounter{equation}{0}
	\setcounter{conjecture}{0}
	\setcounter{remark}{0}
	\setcounter{corollary}{0}
	
	We shall prove our first theorem in five steps.

	\subsection{Bounds for some mixed exponential sums over finite fields}

	Mixed exponential sums over finite fields involve both multiplicative and additive characters. Let $\chi\in\widehat{\mathbb{F}_p^*}$ with $\ord(\chi)=m>1$ and let $\psi$ be an additive character of $\mathbb{F}_p$. Let $X$ be a nonsingular projective algebraic curve defined over $\mathbb{F}_p$ of genus $g_{X}$. Let $f,g\in\mathbb{F}_p(X)$ be two rational functions on $X$ satisfying $f\neq h^m$ and $g\neq h^p-h$ for any $h\in\overline{\mathbb{F}_p}(X)$. To state the following results, we briefly introduce some notations. The divisor of $f$ is defined by the formal sum
	$$\div(f)=\sum_{P\in X(\overline{\mathbb{F}_p})}\ord_{P}(f)\cdot P,$$
	where $\ord_{P}(f)$ is the order of $f$ at point $P$. Let 
	$$\div_{\infty}(f)=\sum_{\substack{P\in X(\overline{\mathbb{F}_p})\\ \ord_{P}(f)<0}}-\ord_{P}(f)\cdot P$$
	be the polar part of $\div(f)$. Also, set 
	$$\supp(\div(f))=\left\{P\in X(\overline{\mathbb{F}_p}): \ord_{P}(f)\neq 0\right\},$$
	and 
	$$\supp(\div_{\infty}(f))=\left\{P\in X(\overline{\mathbb{F}_p}): \ord_{P}(f)<0\right\}.$$
	In addition, the degree of $\div_{\infty}(f)$ is defined by 
	$$\deg(f)_{\infty}=\sum_{\substack{P\in X(\overline{\mathbb{F}_p})\\ \ord_{P}(f)<0}}-\ord_{P}(f).$$

	Castro and Moreno \cite{CM} improved Perel'muter's result \cite{P} and obtained the following useful bound (for consistency with the notation used throughout this paper, we have interchanged the roles of 
	$f$ and $g$ from their original setting). 
	
	\begin{theorem}[Castro and Moreno]\label{Thm. CM bound on mix exponential sums}
		Let notations be as above. Then 
		$$\left|\sum_{\substack{P \in X(\mathbb{F}_p)\\ P\not\in\supp(\div_{\infty}(f))\cup \supp(\div_{\infty}(g))}}\chi(f(P))\psi(g(P))\right|\le \left(2g_{X}-2+s+l+\deg(g)_{\infty}-r\right)\sqrt{p},$$
		where $s=\#\supp(\div(f))$, $l$ is the number of distinct poles of $g$, and $r$ is the number of closed points in $\supp(\div(f))\cap\supp(\div_{\infty}(g))$. 
	\end{theorem}
	
	Now we apply Castro and Moreno's result to our problem. Let 
	$$X=\mathbb{P}^1(\overline{\mathbb{F}_p})=\left\{[x:y]: x,y\in\overline{\mathbb{F}_p}\ \text{and}\ (x,y)\neq (0,0)\right\}$$ be the projective line with the genus $g_{X}=0$. Let $\infty=[1:0]$, and identify the point $[x:1]$ with the affine point $x$ for any $x\in\mathbb{F}_p$. 
	
	Let $f(t)\in\mathbb{F}_p[t]$ with $\deg(f)=n\ge1$ such that $f(t)$ has no multiple roots in $\overline{\mathbb{F}_p}$ and let $g(t)=\xi t$ with $\xi\in\mathbb{F}_p^*$. Let 
	$$f(t)=a_n(t-\alpha_1)(t-\alpha_2)\cdots(t-\alpha_n)$$
	be the decomposition of $f$ in $\overline{\mathbb{F}_p}[t]$. Then one can verify that 
	\begin{align*}
		\div(f^h)&=\alpha_1+\alpha_2+\cdots+\alpha_n-n\cdot\infty,\\
		\div(g^h)&=0-\infty,
	\end{align*}
	where $f^h=f(t_1/t_2), g^h=g(t_1/t_2)\in\mathbb{F}_p(X)$. By adopting the notations in Theorem \ref{Thm. CM bound on mix exponential sums}, we immediately see that $s=n+1$, $l=1$, $\deg(g)_{\infty}=1$ and $r=1$. Clearly $e_p: \mathbb{F}_p\rightarrow\mathbb{C}^*$ defined by $e_p(z)=e^{2\pi \i z/p}$ is an additive character of $\mathbb{F}_p$, where $\i$ is a primitive $4$-th root of unity with argument $\pi/2$. Note that 
	$$\sum_{\substack{P \in X(\mathbb{F}_p)\\ P\not\in\supp(\div_{\infty}(f^h))\cup \supp(\div_{\infty}(g^h))}}\chi(f^h(P))e_p(g^h(P))=\sum_{x\in\mathbb{F}_p}\chi(f(x))e_p(\xi\cdot x).$$
	
We briefly verify the two non-degeneracy conditions required for the application of the Castro--Moreno theorem. Since every zero of $f^h$ is simple, the rational function $f^h\neq u^m$ for any $u\in\overline{\mathbb F}_p(X)$. Indeed, every zero and pole of an $m$-th power has order divisible by $m$, whereas $f^h$ has a zero of order $1$.

Moreover, $g^h=\xi t_1/t_2$ has a simple pole at infinity. On the other hand, if $u\in\overline{\mathbb F}_p(X)$ has a pole of order $r>0$, then $u^p-u$ has a pole of order $pr$, and hence every pole of $u^p-u$ has order divisible by $p$. Therefore,
\begin{align*}
\xi t\neq u^p-u
\qquad
\text{for any}\ u\in\overline{\mathbb F}_p(X).
\end{align*}
Thus both non-degeneracy hypotheses of Theorem \ref{Thm. CM bound on mix exponential sums} are satisfied.	Based on the above discussions, we obtain the following result.

    \begin{lemma}\label{Lem. mix exponential sums}
    	Let $p$ be a prime. Let $\chi\in\widehat{\mathbb{F}_p^*}$ with $\ord(\chi)=m>1$, and let $f(t)\in\mathbb{F}_p[t]$ with $\deg(f)=n\ge1$ such that $f(t)$ has no multiple roots in $\overline{\mathbb{F}_p}$. Then, for any $\xi\in\mathbb{F}_p^*$ we have 
    	$$\left|\sum_{x\in\mathbb{F}_p}\chi(f(x))e_p(\xi\cdot x)\right|\le n\sqrt{p}.$$
    \end{lemma}

    \subsection{Bounds for the Fourier coefficients}

	We begin with the following lemma involving the characteristic function of $D_k(\mathbb{F}_p)$.
	
	\begin{lemma}\label{Lem. characteristic function of kth power}
		Let $k\ge2$ be an integer and $p\equiv 1\pmod {k}$ be a prime with $p-1=km$. Then, for any generator $\chi_p\in\widehat{\mathbb{F}_p^*}$ we have 
		$$1_{D_k(\mathbb{F}_p)}(x)=\frac{1}{k}\sum_{j=0}^{k-1}\chi_p^{mj}(x)=\begin{cases}
			1 & \mbox{if}\ x\in D_k(\mathbb{F}_p),\\
			0 & \mbox{otherwise},
		\end{cases}$$
		where $\chi_p^0=\varepsilon$ is the trivial character. 
	\end{lemma}
	
	\begin{proof}
		Fix a generator $g$ of $\mathbb{F}_p^*$. Since $\chi_p$ is a generator of $\widehat{\mathbb{F}_p^*}$, the number $\chi_p(g)$ is a primitive $(p-1)$-th root of unity. From this, $\chi_p^m(g^j)=1$ if and only if $j\equiv 0\pmod {k}$. In other words, for any $x\in\mathbb{F}_p^*$ we have 
		$$\chi_p^m(x)=1 \Longleftrightarrow x\in D_k(\mathbb{F}_p).$$ 
		Hence the result holds trivially for $x\in D_k(\mathbb{F}_p)\cup\{0\}$. On the other hand, for any $x\in\mathbb{F}_p^*\setminus D_k(\mathbb{F}_p)$, 
		$$\frac{1}{k}\sum_{j=0}^{k-1}\chi_p^{mj}(x)=\frac{1-\chi_p^{mk}(x)}{k(1-\chi_p^m(x))}=\frac{1-\chi_p^{p-1}(x)}{k(1-\chi_p^m(x))}=0.$$
		This proves the lemma. 
	\end{proof}
	
    Recall that $f(t)\in\mathbb{Z}[t]$ is a nonconstant polynomial with $\Delta(f)\neq 0$ and $k\ge2$ is an integer. For any prime $p\equiv 1\pmod{k}$, define the subset 
    \begin{equation}\label{Eq. definition of the set Rp}
    	R_p=\left\{x\in\mathbb{F}_p: f(x)\in D_k(\mathbb{F}_p)\right\}.
    \end{equation}
	By Lemma \ref{Lem. characteristic function of kth power} we have 
	\begin{equation}\label{Eq. chracteristic function of Rp}
		1_{R_p}(x)=\frac{1}{k}\sum_{j=0}^{k-1}\chi_p^{mj}(f(x))=\begin{cases}
			1 & \mbox{if}\ x\in R_p,\\
			0 & \mbox{otherwise}.
		\end{cases}
	\end{equation}
	
	We can obtain the asymptotic formula for $N(p,k,f)$ by (\ref{Eq. chracteristic function of Rp}). To do this, we first introduce the well-known Weil theorem (cf. \cite[Theorem 5.41]{LN}).
	
	\begin{lemma}\label{Lem. the Weil Bound}
		Let $\chi\in\widehat{\mathbb{F}_p^*}$ with $\ord(\chi)=d>1$, and let $f(t)\in\mathbb{F}_p[t]$ be a monic polynomial with $f(t)\neq g(t)^d$ for any $g(t)\in\mathbb{F}_p[t]$. Then, for any $a\in\mathbb{F}_p$ we have 
		$$\left|\sum_{x\in\mathbb{F}_p}\chi(af(x))\right|\le (r-1)\sqrt{p},$$
		where $r$ is the number of distinct roots of $f(t)$ in an algebraic closure $\overline{\mathbb{F}_p}$.
	\end{lemma}
	
	The next result establishes an asymptotic formula for $N(p,k,f)$.
	
	\begin{lemma}\label{Lem. asymptotic formula for N(p,k,f)}
		Let $f(t)\in\mathbb{Z}[t]$ with $\deg(f)=n\ge 1$, leading coefficient $a_n$ and $\Delta(f)\neq 0$. Let $k\ge2$ be an integer and let $p\equiv 1\pmod {k}$ be a prime with $a_n\cdot \Delta(f)\not\equiv 0\pmod{p}$. Then we have the asymptotic formula
		$$N(p,k,f)=\frac{p}{k}+O_{k,f}(\sqrt{p}).$$
	\end{lemma}
	
	\begin{proof}
		By (\ref{Eq. chracteristic function of Rp}) one can verify that 
		\begin{align}\label{Eq. computation of N(p,k,f)}
			N(p,k,f)
&=\sum_{x\in\mathbb{F}_p}1_{R_p}(x)\notag\\
&=\frac{1}{k}\sum_{x\in\mathbb{F}_p}\sum_{j=0}^{k-1}\chi_p^{mj}(f(x))\notag\\
&=\frac{1}{k}\sum_{x\in\mathbb{F}_p}\varepsilon(f(x))+\frac{1}{k}\sum_{j=1}^{k-1}\sum_{x\in\mathbb{F}_p}\chi_p^{mj}(f(x))\notag\\
&=\frac{1}{k}\left(p-\#\left\{x\in\mathbb{F}_p: f(x)=0\right\}\right)+\frac{1}{k}\sum_{j=1}^{k-1}\sum_{x\in\mathbb{F}_p}\chi_p^{mj}(f(x)).
		\end{align}
	   Since $a_n\cdot \Delta(f)\not\equiv 0\pmod {p}$, the reduction of $f$ modulo $p$ is a polynomial over $\mathbb{F}_p$ of degree $n$ and has no multiple roots in $\overline{\mathbb{F}_p}$. Hence, applying Lemma \ref{Lem. the Weil Bound} and (\ref{Eq. computation of N(p,k,f)}), we obtain 
		$$N(p,k,f)\ge \frac{p-n}{k}-\frac{1}{k}\sum_{j=1}^{k-1}\left| \sum_{x\in\mathbb{F}_p}\chi_p^{mj}(f(x)) \right|\ge \frac{p-n}{k}-\frac{(k-1)(n-1)}{k}\sqrt{p},$$
	   and
	   $$N(p,k,f)\le  \frac{p}{k}+\frac{1}{k}\sum_{j=1}^{k-1}\left| \sum_{x\in\mathbb{F}_p}\chi_p^{mj}(f(x)) \right|\le \frac{p}{k}+\frac{(k-1)(n-1)}{k}\sqrt{p}.$$
		This clearly implies that 
		$$N(p,k,f)=\frac{p}{k}+O_{k,f}(\sqrt{p}).$$
		
		In view of the above, we have completed the proof. 
	\end{proof}

	Discrete Fourier analysis on finite abelian groups has extensive applications in additive combinatorics. Readers may refer to \cite[Chapter 4]{TV} for a comprehensive introduction to this topic. We adopt the notations in \cite[Definition 4.6]{TV}.
	
	The Fourier transform of $1_{R_p}$ is a function $\widehat{1_{R_p}}: \mathbb{F}_p\rightarrow\mathbb{C}$ defined by 
	\begin{equation}\label{Eq. the Fourier transform of 1R}
			\widehat{1_{R_p}}(\xi)=\frac{1}{p}\sum_{x\in\mathbb{F}_p}1_{R_p}(x)\overline{e_p(\xi\cdot x)}=\frac{1}{p}\sum_{x\in\mathbb{F}_p}1_{R_p}(x)e_p(-\xi\cdot x)\quad (\forall \xi\in\mathbb{F}_p).
	\end{equation}
	It is known that $1_{R_p}$ has the Fourier expansion (cf. \cite[(4.4)]{TV}) 
	$$1_{R_p}(x)=\sum_{\xi\in\mathbb{F}_p}\widehat{1_{R_p}}(\xi)e_p(\xi\cdot x)\quad (\forall x\in\mathbb{F}_p).$$
	Hence, $\widehat{1_{R_p}}(\xi)$ is also called the Fourier coefficient of $1_{R_p}$. The next result gives an upper bound for the Fourier coefficient $\widehat{1_{R_p}}(\xi)$.
	
	\begin{lemma}\label{Lem. bounds for the Fourier coefficients}
		 Under the same conditions as in Lemma \ref{Lem. asymptotic formula for N(p,k,f)},  for any $\xi\in\mathbb{F}_p^*$, we have 
		 $$\left|\widehat{1_{R_p}}(\xi)\right|\le \frac{n}{pk}\left(1+(k-1)\sqrt{p}\right).$$
	\end{lemma}
	
	\begin{proof}
		By (\ref{Eq. chracteristic function of Rp}) and (\ref{Eq. the Fourier transform of 1R}), for any $\xi\in\mathbb{F}_p^*$ one can verify that 
		\begin{align}\label{Eq. Fourier coefficient is S0+Sj}
			\widehat{1_{R_p}}(\xi)
&=\frac{1}{p}\sum_{x\in\mathbb{F}_p}1_{R_p}(x)e_p(-\xi\cdot x)\notag\\
&=\frac{1}{kp}\sum_{x\in\mathbb{F}_p}\sum_{j=0}^{k-1}\chi_p^{mj}(f(x))e_p(-\xi\cdot x)\notag\\
&=\frac{1}{kp}\sum_{j=0}^{k-1}S_j,
		\end{align}
where 
$$S_j=\sum_{x\in\mathbb{F}_p}\chi_p^{mj}(f(x))e_p(-\xi\cdot x)$$
for any $0\le j\le k-1$. 

We first consider $S_0$. For any $\xi\in\mathbb{F}_p^*$, it is clear that 
$$\sum_{x\in\mathbb{F}_p}e(-\xi\cdot x)=\sum_{x\in\mathbb{F}_p}e_p(x)=0.$$
Thus, one can verify that 
       \begin{align*}
       	S_0
&=\sum_{x\in\mathbb{F}_p}\varepsilon(f(x))e_p(-\xi\cdot x)\\
&=\sum_{x\in\mathbb{F}_p}e_p(-\xi\cdot x)-\sum_{\substack{x\in\mathbb{F}_p\\ f(x)=0}}e_p(-\xi\cdot x)\\
&=-\sum_{\substack{x\in\mathbb{F}_p\\ f(x)=0}}e_p(-\xi\cdot x).
       \end{align*}
        From this we obtain 
       \begin{equation}\label{Eq. bound for S0}
       	|S_0|\le \deg(f)=n.
       \end{equation}

     Next we consider $S_j$ for $1\le j\le k-1$. Since $a_n\cdot \Delta(f)\not\equiv 0\pmod {p}$, the reduction of $f$ modulo $p$ is a polynomial over $\mathbb{F}_p$ of degree $n\ge1$ and has no multiple roots in $\overline{\mathbb{F}_p}$. Noting that $\chi_p^{mj}\neq\varepsilon$ for any $1\le j\le k-1$, by Lemma \ref{Lem. mix exponential sums} we obtain 
     \begin{equation}\label{Eq. bound for Sj}
     	|S_j|\le n\sqrt{p}
     \end{equation}
    for any $1\le j\le k-1$.
    
    Combining (\ref{Eq. bound for S0}) and (\ref{Eq. bound for Sj}) with (\ref{Eq. Fourier coefficient is S0+Sj}), we obtain 
    $$\left|\widehat{1_{R_p}}(\xi)\right|\le \frac{1}{pk}\sum_{j=0}^{k-1}|S_j|\le \frac{n}{pk}\left(1+(k-1)\sqrt{p}\right).$$
    This completes the proof. 
	\end{proof}

	\subsection{The expander mixing lemma}
	
	Graph theory plays an important role in the study of additive combinatorics. We adopt the notations in \cite[Chapter 6]{TV}. Recall that $f(t)\in\mathbb{Z}[t]$ is a nonconstant polynomial with $\Delta(f)\neq 0$ and $k\ge2$ is an integer. For any prime $p\equiv 1\pmod{k}$, we define a simple undirected graph $G_p$ with vertex set $\mathbb{F}_p=\{x_1,x_2,\cdots,x_p\}$, where $x_p=0$. Two different vertices $x,y$ are adjacent if and only if $x+y\in R_p$, where $R_p$ is defined by (\ref{Eq. definition of the set Rp}). 
	Also, we let $M_p$ denote the adjacency matrix of the graph $G_p$. Clearly 
	$$\frac{\#R_p}{p}=\frac{N(p,k,f)}{p}$$
	is the expected edge density of graph $G_p$.  Let 
	$$\left\|  M_p-\frac{N(p,k,f)}{p}J_p   \right\|,$$
	where $J_p$ is a $p\times p$ matrix with all entries $1$, and 
	$$\|M\|=\sup\left\{\frac{|M\u|}{|\u|}: \u\in\mathbb{C}^p\setminus\{\0\}\right\}$$
	denotes the operator norm of a $p\times p$ matrix $M$. Roughly speaking, this quantity measures how far the adjacency matrix $M_p$ is from the matrix corresponding to a perfectly uniform graph. 
	
	Now we have the following result.
	
    \begin{lemma}\label{Lem. the operator norm of adjacency matrix}
    	Under the same conditions as in Lemma \ref{Lem. asymptotic formula for N(p,k,f)}, we have 
    	 $$\left\|  M_p-\frac{N(p,k,f)}{p}J_p   \right\|\le 1+\frac{n}{k}\left(1+(k-1)\sqrt{p}\right)=\Lambda(p,k,f),$$
    	 where $N(p,k,f)$ and $\Lambda(p,k,f)$ are defined by (\ref{Eq. definition of N(p,k,f)}) and (\ref{Eq. definition of Lambda(p,k,f)}) respectively.
    \end{lemma}
	
	\begin{proof} 
		Noting that $x_1,x_2,\cdots,x_p$ is a permutation of $0,1,2\cdots,p-1\in\mathbb{F}_p$, by the formula of the Vandermonde determinant, it is easy to see that 
		\begin{align*}
		  &\det \begin{bmatrix}
			e_p(x_1\cdot x_1)  &  e_p(x_2\cdot x_1)  &    \cdots   &  e_p(x_p\cdot x_1)\\
			e_p(x_1\cdot x_2)  &  e_p(x_2\cdot x_2)  &    \cdots   &  e_p(x_p\cdot x_2)\\
			\vdots                & \vdots                  & \ddots     & \vdots\\
			e_p(x_1\cdot x_p)  &  e_p(x_2\cdot x_p)  &    \cdots   &  e_p(x_p\cdot x_p)
		\end{bmatrix}\\
		=&\pm \det \begin{bmatrix}
			1                                 &  1                                    &    \cdots   &  1\\
			e_p(1\cdot 1)         &  e_p(2\cdot 1)           &    \cdots    &  e_p((p-1)\cdot 1)\\
			\vdots                      & \vdots                           & \ddots       & \vdots\\
			e_p(1\cdot (p-1))  &  e_p(2\cdot (p-1))    &    \cdots     &  e_p((p-1)\cdot (p-1))
		\end{bmatrix}\\
		=&\pm\prod_{1\le i<j\le p}\left(e_p(j)-e_p(i)\right)\\
		\neq& 0.
	\end{align*} 
		From this, the vectors $\e_1,\e_2,\cdots,\e_p\in\mathbb{C}^p$ are linearly independent over $\mathbb{C}$ and hence form a basis of $\mathbb{C}^p$, where 
		$$\e_r=\frac{1}{\sqrt{p}}(e_p(x_r\cdot x_1), e_p(x_r\cdot x_2),\cdots e_p(x_r\cdot x_p))^T$$
		for any $1\le r\le p$. In particular, recalling that $x_p=0$, we have 
		$$\e_p=\frac{1}{\sqrt{p}}(1,1,\cdots,1)^T.$$
		Moreover, for any $1\le i<j\le p$, one can verify that 
		\begin{align*}
			\sum_{r=1}^p \e_i(r,1)\overline{\e_j(r,1)}
&=\frac{1}{p}\sum_{r=1}^p e_p((x_i-x_j)\cdot x_r)\\
&=\frac{1}{p}\sum_{x\in\mathbb{F}_p}e_p(x)\\
&=0.
		\end{align*}
	Combining this with the above discussions, we see that 
	$$\e_1, \e_2, \cdots, \e_p$$
	form an orthonormal basis of $\mathbb{C}^p$. 
		
		Let $W_p$ be a $p\times p$ matrix with $W_p(i,j)=1_{R_p}(x_i+x_j)$ for any $1\le i,j\le p$. We first consider $W_p\e_p$. For any $1\le i\le p$, one can verify that 
		\begin{equation*}
			\sum_{j=1}^{p}W_p(i,j)\e_p(j,1)=\frac{1}{\sqrt{p}}\sum_{j=1}^{p}1_{R_p}(x_i+x_j)=\frac{1}{\sqrt{p}}\sum_{x\in\mathbb{F}_p}1_{R_p}(x)=\frac{\# R_p}{\sqrt{p}}=N(p,k,f)\e_p(i,1).
		\end{equation*}
		This implies that 
		\begin{equation}\label{Eq. actions of Wp on ep}
			W_p\e_p=N(p,k,f)\e_p.
		\end{equation}
		
		Now we focus on $W_p\e_r$ for $r\in(0,p)$. Given any integer $1\le i\le p$, by (\ref{Eq. the Fourier transform of 1R}) and noting that 
		$$e_p(-x_r\cdot x_i)=\overline{e_p(x_r\cdot x_i)},$$
		one can verify that 
		\begin{align*}
			\sum_{j=1}^pW_p(i,j)\e_r(j,1)
&=\sum_{j=1}^{p}1_{R_p}(x_i+x_j)e_p(x_r\cdot x_j)\\
&=\frac{1}{\sqrt{p}}\sum_{x\in\mathbb{F}_p}1_{R_p}(x)e_p(x_r\cdot (x-x_i))\\
&=\frac{e_p(-x_r\cdot x_i)}{\sqrt{p}}\sum_{x\in\mathbb{F}_p}1_{R_p(x)}e_p(x_r\cdot x)\\
&=p\cdot \widehat{1_{R_p}}(-x_r)\overline{\e_r(i,1)}.
		\end{align*}
From this we immediately obtain 
\begin{equation}\label{Eq. actiions of Wp on er}
	W_p\e_r=p\cdot \widehat{1_{R_p}}(-x_r)\overline{\e_r}.
\end{equation}		
	Recall that $J_p$ is a $p\times p$ matrix with all entries $1$. For any $1\le r\le p-1$, we have 
	\begin{equation*}
		\sum_{j=1}^pJ_p(i,j)\e_r(j,1)=\frac{1}{\sqrt{p}}\sum_{j=1}^p\e_r(j,1)=\frac{1}{\sqrt{p}}\sum_{j=1}^p e_p(x_r\cdot x_j)=\frac{1}{\sqrt{p}}\sum_{x\in\mathbb{F}_p}e_p(x)=0.
	\end{equation*}
	This implies $J_p\e_r=\0$ for any $1\le r\le p-1$.  Thus, given an arbitrary vector 
	$$\u=y_1\e_1+y_2\e_2+\cdots+y_{p-1}\e_{p-1}+y_p\e_p\in\mathbb{C}^p\setminus\{\0\},$$
	assembling (\ref{Eq. actions of Wp on ep}) and (\ref{Eq. actiions of Wp on er}) gives
	\begin{align*}
		\left(W_p-\frac{N(p,k,f)}{p}J_p\right)\u
&=\sum_{j=1}^py_j	\left(W_p-\frac{N(p,k,f)}{p}J_p\right)\e_j\\
&=\sum_{j=1}^{p-1}y_jp\widehat{1_{R_p}}(-x_j)\overline{\e_j}.
	\end{align*}
	Since $\e_1,\e_2,\cdots,\e_p$ is an orthonormal basis of $\mathbb{C}^p$, so is $\overline{\e_1}, \ldots, \overline{\e_p}$. Hence \eqref{Eq. actiions of Wp on er} and Lemma \ref{Lem. bounds for the Fourier coefficients} imply 
	\begin{align*}
		\left|\left(W_p-\frac{N(p,k,f)}{p}J_p\right)\u\right|
&=p\left|\sum_{j=1}^{p-1}y_j\widehat{1_{R_p}}(-x_j)\overline{\e_j}\right|\\
&=p\left(\sum_{j=1}^{p-1}|y_j|^2\left|\widehat{1_{R_p}}(-x_j)\right|^2\right)^{1/2}\\
&\le \frac{n}{k}\left(1+(k-1)\sqrt{p}\right)\cdot \left(\sum_{j=1}^{p-1}|y_j|^2\right)^{1/2}\\
&\le \frac{n}{k}\left(1+(k-1)\sqrt{p}\right)|\u|,
	\end{align*}
which yields 
	\begin{equation}\label{Eq. operator norm for Wp-J}
		\left\|  W_p-\frac{N(p,k,f)}{p}J_p \right\|\le \frac{n}{k}\left(1+(k-1)\sqrt{p}\right).
	\end{equation}
	
	On the other hand, let the diagonal matrix 
	$$D_p=\diag\left(1_{R_p}(2x_1), 1_{R_p}(2x_2), \cdots, 1_{R_p}(2x_p)\right).$$
	Since $1_{R_p}(2x_j)\in\{0,1\}$, for any vector $\u=(u_1,u_2,\cdots,u_p)^T\in\mathbb{C}^p$ we have 
	\begin{align*}
		\left|D_p \u\right|=\left(\sum_{j=1}^p |u_j|^21_{R_p}(2x_j)^2\right)^{1/2}\le \left(\sum_{j=1}^p |u_j|^2\right)^{1/2}\le |\u|.
	\end{align*}
	Thus, 
	\begin{equation}\label{Eq. operator norm for Dp}
		\left\| D_p  \right\|\le 1.
	\end{equation}
		As the adjacency matrix $M_p=W_p-D_p$, assembling (\ref{Eq. operator norm for Wp-J}) and (\ref{Eq. operator norm for Dp}) gives 
		\begin{align*}
	 \left\|  M_p-\frac{N(p,k,f)}{p}J_p  \right\|
&=\left\|  W_p-\frac{N(p,k,f)}{p}J_p-D_p \right\| \\
&\le \left\|  W_p-\frac{N(p,k,f)}{p}J_p\right\|+\left\|  D_p \right\| \\
&\le 1+\frac{n}{k}\left(1+(k-1)\sqrt{p}\right).
		\end{align*}
	
	In view of the above, we have completed the proof. 
	\end{proof}
	
	The expander mixing lemma mainly concerns the distribution of edges in a graph. Readers may refer to the survey papers \cite{Chung, HLW} and the book \cite[Corollary 9.2.5]{AS} for a comprehensive introduction to this topic.
	
	Recall that $\mathbb{F}_p=\{x_1,x_2,\cdots,x_p\}$, and $M_p$ is the adjacency matrix of the graph $G_p$. For any nonempty subsets $A,B\subseteq\mathbb{F}_p$, define 
	\begin{equation}\label{Eq. definition of eG}
			e_G(A,B)=\sum_{\substack{(a,b)\in A\times B\\ a\neq b}}1_{R_p}(a+b)=\sum_{\substack{1\le i\le p\\ x_i\in A}}\sum_{\substack{1\le j\le p\\ x_j\in B}}M_p(i,j).
	\end{equation}

	Recall that 
	$$	\Lambda(p,k,f)=1+\frac{\deg(f)}{k}(1+(k-1)\sqrt{p})$$
	is defined by (\ref{Eq. definition of Lambda(p,k,f)}). 	Now we establish the expander mixing lemma for graph $G_p$. 
	
	\begin{lemma}\label{Lem. The expander mixing lemma}
			Under the same conditions as in Lemma \ref{Lem. asymptotic formula for N(p,k,f)}, for any nonempty subsets $A,B\subseteq\mathbb{F}_p$, we have 
		 $$\left|e_G(A,B)-\frac{N(p,k,f)}{p}\#A\#B\right|\le\Lambda(p,k,f)\sqrt{\#A\#B}.$$
	\end{lemma}
	
	\begin{proof}
		For any nonempty subset $U\subseteq\mathbb{F}_p$, let the $p\times 1$ matrix
		$$\1_{U}=\left(1_U(x_1),1_U(x_2),\cdots,1_U(x_p)\right)^T.$$
		Then one can verify that $\1_A^TM_p\1_B=e_G(A,B)$ and $\1_A^TJ_p\1_B =\#A\#B$. Thus, applying Lemma \ref{Lem. the operator norm of adjacency matrix} and the Cauchy–Schwarz inequality, we obtain 
		\begin{align*}
			\left|e_G(A,B)-\frac{N(p,k,f)}{p}\#A\#B\right|
&=\left|\1_A^T\left(M_p-\frac{N(p,k,f)}{p}J_p\right)\1_B\right|\\
&\le \left|\1_A\right|\cdot \left|\left(M_p-\frac{N(p,k,f)}{p}J_p\right)\1_B\right|\\
&\le  \left\|  M_p-\frac{N(p,k,f)}{p}J_p  \right\| \cdot \left|\1_A\right|\cdot \left|\1_B\right|\\
&\le \Lambda(p,k,f)\sqrt{\#A\#B}.
		\end{align*}
		
	In view of the above, we have completed the proof. 
	\end{proof}
	
	\subsection{The independence number and vertex connectivity}
	
	Let $A\subseteq\mathbb{F}_p$ be a nonempty subset. Let $H_A$ be the subgraph induced by $A$. In other words, the vertex set of $H_A$ is $A$ and any two distinct elements $a_1, a_2\in A$ are adjacent in $H_A$ if and only if they are adjacent in $G_p$.
	
	A nonempty subset $I\subseteq A$ is called an independent subset of $H_A$ if $\#I=1$ or any two different elements of $I$ are not adjacent in $H_A$. The independence number $\alpha(H_A)$ is defined by 
	$$\alpha(H_A)=\max\left\{\#I: I\ \text{is an independent subset of $H_A$}\right\}.$$
	
	The next result gives an upper bound for $\alpha(H_A)$.
	
	\begin{lemma}\label{Lem. upper bound for independence number}
		 	Under the same conditions as in Lemma \ref{Lem. asymptotic formula for N(p,k,f)}, for any subset $A\subseteq\mathbb{F}_p$, we have 
		 $$\alpha(H_A)\le   \frac{p\Lambda(p,k,f)}{N(p,k,f)}.$$
	\end{lemma}
	
	\begin{proof}
		Let $I\subseteq A$ be an independent subset of $A$ with $\#I=\alpha(H_A)$. Since $H_A$ is a subgraph induced by $A$, any two distinct elements of $I$ are not adjacent in $G_p$. Hence $e_G(I,I)=0$, where $e_G$ is defined by (\ref{Eq. definition of eG}). Applying Lemma \ref{Lem. The expander mixing lemma}, we obtain 
		$$\left|e_G(I,I)-\frac{N(p,k,f)}{p}(\#I)^2\right|\le \left(1+\frac{n}{k}\left(1+(k-1)\sqrt{p}\right)\right)\#I.$$
		This implies that 
		$$\alpha(H_A)=\#I\le  \frac{p}{N(p,k,f)}\left(1+\frac{n}{k}\left(1+(k-1)\sqrt{p}\right)\right)=\frac{p\Lambda(p,k,f)}{N(p,k,f)}.$$
		This completes the proof.
	\end{proof}
	
	The vertex connectivity $\kappa(H_A)$ is the smallest number of vertices whose deletion disconnects the graph $H_A$ or reduces it to a single vertex, with the convention that $\kappa(H_A)=0$ if $H_A$ is disconnected. 
	
	When $p\equiv 1\pmod {k}$ is sufficiently large, the next result establishes an inequality between $\kappa(H_A)$ and $\alpha(H_A)$. 
	
	\begin{lemma}\label{Lem. inequality for vertex connectivity}
		 	Under the same conditions as in Lemma \ref{Lem. asymptotic formula for N(p,k,f)}, if $p$ is sufficiently large, then for any subset $A\subseteq\mathbb{F}_p$ with 
		 	$$\#A\ge p+2-N(p,k,f)+\left\lfloor\frac{2p\Lambda(p,k,f)}{N(p,k,f)}\right\rfloor,$$
		 	we have 
		 $$\kappa(H_A)\ge\alpha(H_A).$$
	\end{lemma}
	
	\begin{proof}
		Let $\delta(H_A)$ be the minimum degree of $H_A$, i.e., 
		$$\delta(H_A)=\min\left\{\deg_{H_A}(x): x\in A\right\},$$
		where $\deg_{H_A}(x)$ denotes the degree of $x$ in $H_A$. For any $x\in A$, since 
		$$\#A\ge p+2-N(p,k,f)+\left\lfloor\frac{2p\Lambda(p,k,f)}{N(p,k,f)}\right\rfloor,$$ 
		one can verify that 
		\begin{align*}
			\deg_{H_A}(x)
&=\sum_{y\in A\setminus\{x\}}1_{R_p}(x+y)\\
&=\sum_{y\in\mathbb{F}_p}1_{R_p}(x+y)-1_{R_p}(2x)-\sum_{y\in\mathbb{F}_p\setminus A}1_{R_p}(x+y)\\
&\ge \sum_{y\in\mathbb{F}_p}1_{R_p}(y)-(1+p-\#A)\\
&=N(p,k,f)+\#A-(p+1)\\
&\ge N(p,k,f)+p+2-N(p,k,f)+\left\lfloor\frac{2p\Lambda(p,k,f)}{N(p,k,f)}\right\rfloor-(p+1)\\
&=\left\lfloor\frac{2p\Lambda(p,k,f)}{N(p,k,f)}\right\rfloor+1\\
&>\frac{2p\Lambda(p,k,f)}{N(p,k,f)}.
		\end{align*}
    Thus,  
    \begin{equation}\label{Eq. lower bound for delta(HA)}
    	\delta(H_A)>\frac{2p\Lambda(p,k,f)}{N(p,k,f)}.
    \end{equation}
 
    We next show that 
    $$\kappa(H_A)\ge \alpha(H_A).$$
	Suppose, to the contrary, that $\kappa(H_A)<\alpha(H_A)$. From this and Lemma \ref{Lem. upper bound for independence number}, there exists a subset $S$ with 
	\begin{equation}\label{Eq. upper bound for S}
		\#S<\alpha(H_A)\le \frac{p\Lambda(p,k,f)}{N(p,k,f)}
	\end{equation}
	such that $H_A-S$ is disconnected, where $H_A-S$ is the subgraph induced by $A\setminus S$. Thus, for any $x\in A\setminus S$, by (\ref{Eq. lower bound for delta(HA)}), (\ref{Eq. upper bound for S}) and noting that $N(p,k,f)\le p$, we obtain 
	\begin{align*}
	\deg_{H_A-S}(x)\ge \deg_{H_A}(x)-\#S
&\ge \delta(H_A)-\frac{p\Lambda(p,k,f)}{N(p,k,f)}\\
&>\frac{2p\Lambda(p,k,f)}{N(p,k,f)}-\frac{p\Lambda(p,k,f)}{N(p,k,f)}\\
&=\frac{p\Lambda(p,k,f)}{N(p,k,f)}\\
&\ge \Lambda(p,k,f)\\
&\ge \frac{k-1}{k}\sqrt{p}\\
&\ge \frac{1}{2}\sqrt{p}.
    \end{align*}
	This implies that $H_A-S$ has at least two vertices when $p$ is sufficiently large. Hence, $H_A-S$ has at least two connected components. 
	
	Let $C_0$ be a connected component of $H_A-S$ with vertex set $V(C_0)$ such that
	$$\#V(C_0)=\min\left\{\#V(C): C\ \text{is a connected component of $H_A-S$}\right\}.$$ 
	For any $x\in V(C_0)$, since $C_0$ is a connected component of $H_A-S$, each neighbor of $x$ in $H_A$ belongs to $(V(C_0)\setminus\{x\})\cup S$. Thus, 
	$$\delta(H_A)\le \deg_{H_A}(x)\le \#\left((V(C_0)\setminus\{x\})\cup S\right)=\#V(C_0)+\#S-1.$$
	Using (\ref{Eq. lower bound for delta(HA)}) and (\ref{Eq. upper bound for S}), we obtain 
	\begin{equation}\label{Eq. V(C0) is large}
		\#V(C_0)\ge 1+\delta(H_A)-\#S>1+\frac{2p\Lambda(p,k,f)}{N(p,k,f)}-\frac{p\Lambda(p,k,f)}{N(p,k,f)}>\frac{p\Lambda(p,k,f)}{N(p,k,f)}.
	\end{equation}

	Let $B=A\setminus(V(C_0)\cup S)$. Since $H_A-S$ has at least two connected components and $C_0$ is a connected component of $H_A-S$ with the minimum number of vertices, by (\ref{Eq. V(C0) is large}) we immediately obtain 
	$$\#B\ge \#V(C_0)>\frac{p\Lambda(p,k,f)}{N(p,k,f)}.$$
    Moreover, by the definition of $B$, there are no edges between $V(C_0)$ and $B$ in $H_A$. Since $H_A$ is the subgraph of $G_p$ induced by $A$, there are no edges between $V(C_0)$ and $B$ in $G_p$, i.e., 
    $$e_G(V(C_0),B)=0.$$
    Applying Lemma \ref{Lem. The expander mixing lemma} with this, we obtain 
    $$\frac{N(p,k,f)}{p}\#V(C_0)\#B\le \Lambda(p,k,f)\sqrt{\#V(C_0)\#B}.$$
    Since both $V(C_0)$ and $B$ are nonempty, the above inequality yields 
    $$\sqrt{\#V(C_0)\#B}\le \frac{p\Lambda(p,k,f)}{N(p,k,f)}.$$
    This contradicts the fact that $\#V(C_0)$ and $\#B$ are both strictly greater than $\frac{p\Lambda(p,k,f)}{N(p,k,f)}$. Thus, 
    $$\kappa(H_A)\ge \alpha(H_A).$$
    
    In view of the above, we have completed the proof. 		
	\end{proof}
	
	\subsection{Final proof}
	
	In 1979, Chv\'atal and Erd\H{o}s \cite{CE} obtained the following elegant result, which will be used in the proof of our theorem.
	
	\begin{theorem}[Chv\'atal and Erd\H{o}s]\label{Thm. CE}
		Let $G$ be a graph with at least three vertices. If the independence number $\alpha(G)$ and the vertex connectivity $\kappa(G)$ satisfy $\kappa(G)\ge\alpha(G)$, then $G$ has a Hamilton cycle. 
	\end{theorem}

	Now we are in a position to finish the proof of Theorem \ref{Thm. A}.
	
	{\noindent\bfseries Proof of Theorem \ref{Thm. A}}. Since $a_n\Delta(f)\neq0$, there are only finitely many primes dividing $a_n\Delta(f)$. Let $p$ be sufficiently large such that the hypotheses concerning the reduction of $f$ modulo $p$ that are required in Lemmas \ref{Lem. asymptotic formula for N(p,k,f)}--\ref{Lem. inequality for vertex connectivity} are satisfied, and $p\nmid a_n\Delta(f)$.
	
	By Lemma \ref{Lem. inequality for vertex connectivity} we have $\kappa(H_A)\ge\alpha(H_A)$. Thus, applying Theorem \ref{Thm. CE}, there exists a Hamilton cycle 
	$$a_1,a_2,\cdots,a_{\#A}, a_1$$
	in $H_A$. In other words, we have 
	$$f(a_i+a_{i+1})\in D_k(\mathbb{F}_p)$$
	for any $1\le i \le \#A$, where $a_{\#A+1}=a_1$.
	
	In view of the above, we have completed the proof. \qed 
	
	\section{Proof of Theorem \ref{Thm. B}}
	\setcounter{lemma}{0}
	\setcounter{theorem}{0}
	\setcounter{equation}{0}
	\setcounter{conjecture}{0}
	\setcounter{remark}{0}
	\setcounter{corollary}{0}
	
	Let notations be as above. Now we state the proof of our second theorem.
	
	{\noindent\bfseries Proof of Theorem \ref{Thm. B}}. By Theorem \ref{Thm. A} we immediately obtain 
	$$c_{\min}(p,k,f)\le  p+2-N(p,k,f)+\left\lfloor\frac{2p\Lambda(p,k,f)}{N(p,k,f)}\right\rfloor.$$
	
	Next we focus on the lower bound for $c_{\min}(p,k,f)$. It follows from Lemma \ref{Lem. asymptotic formula for N(p,k,f)} that 
	$$N(p,k,f)=\#R_p=\#\left\{x\in\mathbb{F}_p: f(x)\in D_k(\mathbb{F}_p)\right\}=\frac{p}{k}+O_{k,f}(\sqrt{p}),$$
	where $R_p$ is defined by (\ref{Eq. definition of the set Rp}). Hence $R_p\neq\emptyset$ for every sufficiently large $p$ with $p\equiv 1\pmod{k}$. Since $p>2$, we have $1/2\in\mathbb{F}_p$. Fix an element $r_0\in R_p$ and let $x_0=r_0/2$. Then 
	\begin{align*}
		\deg_{G_p}(x_0)
&=\sum_{\substack{y\in\mathbb{F}_p\\ y\neq x_0}}1_{R_p}(x_0+y)\\
&=\sum_{y\in\mathbb{F}_p}1_{R_p}(x_0+y)-1_{R_p}(2x_0)\\
&=\sum_{y\in\mathbb{F}_p}1_{R_p}(y)-1\\
&=N(p,k,f)-1.
	\end{align*}
	Let the subset
	$$A_0=\{x_0\}\cup\left\{y\in\mathbb{F}_p: 1_{R_p}(x_0+y)=0\right\}.$$
	Then, by Lemma \ref{Lem. asymptotic formula for N(p,k,f)} one can verify that 
	\begin{equation}\label{Eq. number of elements in A0}
		\#A_0=p-\deg_{G_p}(x_0)=p+1-N(p,k,f)=\left(1-\frac{1}{k}\right)p+O_{k,f}(\sqrt{p}).
	\end{equation}
	Since $k\ge2$, we have $\#A_0\ge 3$ when $p\equiv 1\pmod{k}$ is sufficiently large. 
	
	Let  $H_{A_0}$ be the subgraph induced by $A_0$. Note that $x_0$ is an isolated vertex in $H_{A_0}$, i.e., $\deg_{H_{A_0}}(x_0)=0$. Hence the graph $H_{A_0}$ has no Hamilton cycles. In other words, there is no permutation of $A_0$ satisfying the required conditions. Thus, by (\ref{Eq. number of elements in A0}) we have 
	\begin{equation}\label{Eq. inequality in the proof of Thm. B}
		p+2-N(p,k,f)\le c_{\min}(p,k,f)\le  p+2-N(p,k,f)+\left\lfloor\frac{2p\Lambda(p,k,f)}{N(p,k,f)}\right\rfloor.
	\end{equation}
	Noting that $N(p,k,f)\asymp_{k,f}p$ by Lemma \ref{Lem. asymptotic formula for N(p,k,f)} and that $\Lambda(p,k,f)\asymp_{k,f}p^{1/2}$, we have 
	$$\left\lfloor\frac{2p\Lambda(p,k,f)}{N(p,k,f)}\right\rfloor\asymp \frac{2p\Lambda(p,k,f)}{N(p,k,f)}\asymp_{k,f}p^{1/2}.$$
	From this and Lemma \ref{Lem. asymptotic formula for N(p,k,f)}, the inequality (\ref{Eq. inequality in the proof of Thm. B}) implies that 
	$$\left(1-\frac{1}{k}\right)p+O_{k,f}(\sqrt{p})\le c_{\min}(p,k,f)\le \left(1-\frac{1}{k}\right)p+O_{k,f}(\sqrt{p}).$$
	Hence, 
	$$c_{\min}(p,k,f)= \left(1-\frac{1}{k}\right)p+O_{k,f}(\sqrt{p}).$$
	
	In view of the above, we have completed the proof. \qed

\end{document}